\documentclass[12pt]{article}

\title{The absolute gradings on embedded contact homology and Seiberg-Witten Floer cohomology}
\author{Daniel Cristofaro-Gardiner}
\date{}

\usepackage{amssymb}
\usepackage{latexsym}
\usepackage{amsmath}
\usepackage{amsthm}
\usepackage{amscd}
\usepackage[dvips]{graphics}

\numberwithin{equation}{section}
\newcommand{\mc}[1]{{\mathcal #1}}
\newtheorem{theorem}{Theorem}[section]
\newtheorem{proposition}[theorem]{Proposition}

\newtheorem{lemma}[theorem]{Lemma}
\newtheorem{lemma-definition}[theorem]{Lemma-Definition}

\theoremstyle{definition}

\newtheorem{remark}[theorem]{Remark}

\newcommand{\eqdef}{\;{:=}\;}

\newcommand{\C}{{\mathbb C}}

\newcommand{\R}{{\mathbb R}}

\newcommand{\Z}{{\mathbb Z}}
\newcommand{\op}{\operatorname}

\newcommand{\Ker}{\op{Ker}}

\newcommand{\tensor}{\otimes}

\newcommand{\bpm}{\begin{pmatrix}}
\newcommand{\epm}{\end{pmatrix}}

\begin{document}

\setcounter{tocdepth}{2}

\maketitle
 
\begin{abstract}
    Let $Y$ be a closed connected contact $3$-manifold.  In \cite{echswf}, Taubes defines an isomorphism between the embedded contact homology (ECH) of $Y$ and its Seiberg-Witten Floer cohomology.  Both the ECH of $Y$ and the Seiberg-Witten Floer cohomology of $Y$ admit absolute gradings by homotopy classes of oriented two-plane fields.  We show that Taubes' isomorphism preserves these gradings.  To do this, we prove another result relating the expected dimension of any component of the Seiberg-Witten moduli space over a completed connected symplectic cobordism to the ECH index of a corresponding homology class.  
\end{abstract}

\section{Introduction}
\label{sec:intro}

Let $Y$ be a closed connected oriented $3$-manifold.  A {\em contact form} on $Y$ is a $1$-form $\lambda$ such that $\lambda \wedge d\lambda > 0$.  A contact form determines the {\em Reeb vector field} $R$ by the equations \[d\lambda(R,\cdot)=0, \quad \lambda(R)=1,\] and an oriented $2$-plane field $\xi \eqdef \op{Ker}(\lambda)$, called the {\em contact structure} for $\alpha$.  A {\em Reeb orbit} is a map $\gamma: \R/T\Z$ for some $T > 0$ such that $\gamma'(t)=R(\gamma(t))$.  A Reeb orbit $\gamma$ is called {\em nondegenerate} if  for some $y$ on the image of $\gamma$ the linearized flow along $\gamma$ restricted to $\xi_y$ does not have $1$ as an eigenvalue.  If $\gamma$ is nondegenerate and the eigenvalues of the linearized flow are real then $\gamma$ is called {\em hyperbolic}; otherwise, $\gamma$ is called {\em elliptic}.  A contact form is called {\em nondegenerate} if all of its Reeb orbits are nondegenerate. 

If $\lambda$ is nondegenerate and $\Gamma \in H_1(Y)$, then the {\em embedded contact homology} $ECH(Y,\lambda,\Gamma)$ of $Y$ is defined.  This is the homology of a chain complex freely generated over $\Z/2$\footnote{Embedded contact homology can also be defined over $\Z$, see \cite[\S 9]{obg2}.} by certain finite sets of Reeb orbits, called {\em orbit sets}, with respect to a differential that counts certain mostly embedded $J$-holomorphic curves in the symplectization of $Y$.  In \cite{echswf}, Taubes defines an isomorphism between ECH and the {\em Seiberg-Witten Floer cohomology} defined by Kronheimer and Mrowka in \cite{km}.  Specifically, Taubes shows \cite[Theorem 1]{taubes1} that there is a canonical isomorphism of relatively graded $\Z/2$ modules
\begin{equation}
\label{eqn:MainIsomorphism}
\mathcal{T}: ECH_*(Y,\lambda,\Gamma) \simeq \widehat{HM}^{-*}(Y,\mathfrak{s}_{\xi}+\op{PD}(\Gamma)),
\end{equation}  
where $\mathfrak{s}_{\xi}$ is a certain spin$^c$ structure determined by $\xi$, see \cite[\S 8]{hutchings_summary}, $\op{PD}(\Gamma)$ denotes the Poincare dual of $\Gamma$, and $\widehat{HM}^{-*}$ denotes the relatively graded module $\widehat{HM}^{*}$ with the grading reversed.  

Both embedded contact homology and Seiberg-Witten Floer cohomology admit absolute gradings by homotopy classes of oriented $2$-plane fields, see \cite{hutchings_absolute} and \cite{km}.  The main theorem of this paper asserts that the map $\mathcal{T}$ preserves this extra structure.  To be explicit, denote the direct sum of $ECH(Y,\lambda,\Gamma)$ over all $\Gamma$ by $ECH(Y,\lambda)$, and denote the direct sum of $\widehat{HM}^{-*}(Y,\mathfrak{s})$ over all isomorphism classes of spin$^c$ structures on $Y$ by $\widehat{HM}^{-*}(Y)$.  Let $j$ be a homotopy class of oriented $2$-plane fields on $Y$, and denote by $ECH_j(Y,\lambda)$ and $\widehat{HM}^{j}(Y)$ the submodules with grading $j$ of $ECH(Y,\lambda)$ and $\widehat{HM}^{-*}(Y)$ respectively.  We show:	
\begin{theorem}
\label{thm:MainTheorem}
The map $\mathcal{T}$ restricts to an isomorphism 
\begin{equation}
\label{eqn:RefinedIsomorphism}
ECH_j(Y,\lambda) \simeq \widehat{HM}^{j}(Y).
\end{equation}  
\end{theorem}  

Theorem~\ref{thm:MainTheorem} implies that the absolutely graded $\Z/2$-module $ECH(Y,\lambda)$ is a topological invariant.  Theorem~\ref{thm:MainTheorem} follows from another result of potentially independent interest relating the expected dimension of any component of the Seiberg-Witten moduli space over a completed connected symplectic cobordism to the ECH index of a corresponding homology class, see Theorem~\ref{thm:swindex} below for the precise statement.

\section{Embedded contact homology}
\label{sec:ech}

We begin by reviewing those aspects of embedded contact homology that are relevant to the proofs of Theorem~\ref{thm:MainTheorem} and Theorem~\ref{thm:swindex}. 

\subsection{Definition of embedded contact homology}

We will first review the definition of embedded contact homology.  Define $ECC(Y,\lambda,\Gamma,J)$ to be the chain complex generated over $\Z/2$ by finite 
sets $\alpha = \lbrace (\alpha_i,m_i)\rbrace$ such that each $\alpha_i$ is a Reeb orbit, $m_i = 1$ if $\alpha_i$ is 
hyperbolic, and 
\[
\sum_i m_i[\alpha_i]= \Gamma \in H_1(Y).
\]  

An $\R$-invariant almost complex structure $J$ is called {\em admissible} if $J$ sends the two-plane field $\xi$ 
to itself, rotating it positively with respect to $d\lambda$, and satisfies $J(\partial_s)=R$, where $s$ denotes the $\R$ 
coordinate on $\R \times Y$.  The ECH chain complex differential $\partial_{ECH}$ counts certain $J$-holomorphic curves in $\R \times Y$ for an admissible $J$.  Specifically, if $\alpha$ and $\beta$ are two chain complex generators, then the coefficient $\langle \partial \alpha,\beta \rangle \in \Z/2$ is a count 
of $J$-holomorphic curves in $\R \times Y$, modulo translation in the $\R$ coordinate, that are asymptotic as currents to $\R \times \alpha$ as $s \to \infty$ and to $\R \times \beta$ as $s \to -\infty$ and which have {\em ECH index\/} $1$.  The ECH index, a certain function of the relative homology class of the curve, will be reviewed in \S\ref{sec:ECHIndex}.  If $J $ is generic, then $\partial$ is well-defined and $\partial^2 = 0$, see \cite{obg1} and \cite{obg2}.  

Define $ECH(Y,\lambda,\Gamma)$ to be the homology of this chain complex.  A priori, this might depend on $J$, but by the canonical isomorphism \eqref{eqn:MainIsomorphism} it does not.  The ECH index induces a relative $\Z/p$ grading on $ECH(Y, \lambda,\Gamma)$, as reviewed in \S\ref{sec:ECHIndex}, where $p$ denotes the divisibility of $c_1(\xi)+2\op{PD}(\Gamma)$ in $H^2(Y)$ mod torsion.   

\subsection{The absolute grading on ECH}
\label{sec:ECHAbsoluteGrading}

The relative $\Z/p$ grading on ECH can be refined to an absolute grading by homotopy classes of oriented $2$-plane fields.  We now review this construction.  For a  review of homotopy classes of oriented $2$-plane fields, see \cite[\S 3.1]{hutchings_absolute} (in particular, note that we follow the sign convention for the $\Z$-action on the set of homotopy classes of oriented $2$-plane fields in \cite[\S 3.1]{hutchings_absolute} by demanding that the isomorphism $\pi_3(S^2) \simeq \Z$ that sends the Hopf fibration to $+1$ is an isomorphism of $\Z$-sets).    

Recall first that a link $\mathcal{L}$ in $Y$ is {\em transversal\/} if $\mathcal{L}$ is transverse to the contact plane field at every point.  Let $\mathcal{L}$ be a transversal link and orient $\mathcal{L}$ so that it intersects the contact plane field positively.  A framing of $\mathcal{L}$ is equivalent to a homotopy class of symplectic trivializations of $\xi|_{\mathcal{L}}$.  Given a transversal link $\mathcal{L}$ with framing $\tau$, we can define a homotopy class of $2$-plane fields which we will denote by $P_{\tau}(\mathcal{L})$.

To do this, begin by taking a tubular neighborhood $N$ of $\mathcal{L}$.  On $N$, choose disjoint tubular neighborhoods $N_K$ for each component $K$ of the link and choose coordinates $\varphi_K: N_K \stackrel {\simeq}{\longrightarrow} S^1 \times D^2$ such that $\varphi_K$ sends $K$ to $S^1 \times \lbrace 0 \rbrace$ and $d\varphi_K$ sends $\xi|_K$ to $0 \times \R^2$ compatibly with $\tau$; extend this trivialization to a trivialization of the tangent bundle such that the contact plane field is identified with $\lbrace 0 \rbrace \times \R^2$ and the Reeb vector field is identified with $(1,0,0)$ at each point.  Next, choose a vector field $P$ such that on $S^1\times\{z\in D^2 \mid |z|>1/2\}$, the vector field $P$ intersects $\xi$ positively, on $S^1\times\{z\in D^2 \mid |z|<1/2\}$ the vector field $P$
intersects $\xi$ negatively, and on $S^1\times\{z\in D^2 \mid |z|=1/2\}$, the vector field $P$ is given according to the above trivialization by 
\begin{equation}
\label{eqn:P12}
P(t,e^{i\theta}/2) \eqdef (0,e^{-i\theta}).
\end{equation}

A homotopy class of vector fields determines a homotopy class of $2$-plane fields.   On $N$, define $P_{\tau}(\mathcal{L})$ to be the $2$-plane field determined by this vector field.   On $Y\setminus N$, set $P_{\tau}(\mathcal{L})$ equal to $\xi$.  This uniquely determines the homotopy class of $P_{\tau}(\mathcal{L}).$   

\begin{remark}
To compare the above construction to a perhaps more familiar one, note that if instead of requiring \eqref{eqn:P12}, we require that 
\[
P(t,e^{i\theta}/2) \eqdef (0,e^{i\theta}),
\]
then the homotopy class of the resulting $2$-plane field corresponds to the contact structure obtained from $\xi$ via a {\em Lutz twist} along $\mathcal{L}$ as defined for example in \cite{geiges}.  In particular, the resulting homotopy class of $2$-plane field does not depend on the framing $\tau$.  In our case, the homotopy class does depend on the framing: if $\tau\prime$ is another trivialization, then    
\[
P_\tau(\mathcal{L}) - P_{\tau'}(\mathcal{L}) \equiv 2(\tau-\tau') \mod d(c_1(\xi) + 2
\op{PD}([\mathcal{L}])).
\] 	
This is explained in \cite[\S 3.3]{hutchings_absolute}. 
\end{remark}

To associate a homotopy class of two-plane fields to an orbit set $\alpha=\lbrace (\alpha_i,m_i) \rbrace$, first choose trivializations $\tau = \lbrace \tau_i \rbrace$ of $\xi$ over each $\alpha_i$.  Next, choose disjoint tubular neighborhoods $N_i$ of the $\alpha_i$.  Finally, in each $N_i$ choose a braid $\zeta_i$ with $m_i$ strands around each $\alpha_i$ (this means that $\zeta_i$ is an oriented link in $N_i$ such that the projection of $\zeta_i$ to $\alpha_i$ is a degree $m$ orientation preserving submersion), and define $\mathcal{L}$ to be the union of these braids, with the framing induced by $\tau$.  Define $I_{ECH}(\alpha)$ by the formula 
\begin{equation}
\label{eqn:ECHabsolutegradingequation}
I_{ECH}(\alpha) \eqdef P_\tau(\mathcal{L}) - \sum_i w_{\tau_i}(\zeta_i) +
  \mu_\tau(\alpha),
\end{equation}
where $w_{\tau_i}(\zeta_i)$ is the writhe of the link $\zeta_i$ with respect to $\tau_i$ as defined in \cite[\S 2.6]{hutchings_absolute}, and $\mu_{\tau}(\alpha)$ is a certain sum of Conley-Zehnder index terms associated to $\alpha$, see \cite[\S 2.8]{hutchings_absolute} for the precise definitions.  

It is shown in \cite[Lem. 3.7]{hutchings_absolute} that $I_{ECH}(\alpha)$ is well-defined.  The homotopy class of $2$-plane fields $I_{ECH}(\alpha)$ is the absolute grading of the generator $\alpha$.   

\subsection{Symplectic cobordisms and the ECH index}
\label{sec:ECHIndex}

The proof of Theorem~\ref{thm:MainTheorem} and the statement of Theorem~\ref{thm:swindex} both involve the ECH index.  We now briefly review this construction.

Let $(Y_+,\lambda_+)$ and $(Y_-,\lambda_-)$ be closed contact $3$-manifolds.  A (connected) {\em symplectic cobordism} from $Y_+$ to $Y_-$ is a connected compact symplectic $4$-manifold $(X,\omega)$ such that $\partial X = -Y_{-} \sqcup Y_+$ and $\omega|_{Y_{\pm}}=d\lambda_{\pm}$.  Given a symplectic cobordism, it is a standard fact that one can always find neighborhoods $N_{\pm}$ of $Y_{\pm}$ in $X$ such that $(N_+,\omega)$ and $(N_-,\omega)$ are symplectomorphic to $((-\epsilon,0] \times Y_+,d(e^s\lambda_+))$ and $([0,\epsilon) \times Y_-,d(e^s\lambda_-))$ respectively.  We can therefore attach cylindrical ends to $(X,\omega)$ to obtain a non-compact symplectic manifold $\overline{X}$ called the {\em symplectic completion of X}.  Specifically, define $E_+ \eqdef [0,\infty) \times Y_+$ and $E_- \eqdef (-\infty,0] \times Y_-$.  Then $(\overline{X},\omega)$ is the symplectic manifold obtained by gluing $E_{\pm}$ to $Y_{\pm}$ via the  above identifications.

Let $X$ be a symplectic cobordism from $Y_+$ to $Y_-$.  If $\alpha^{+} = \lbrace (\alpha_i^{+},m_i^{+}) \rbrace$ is an orbit set in $Y_+$ and $\alpha^{-} = \lbrace (\alpha_{j}^{-},m_{j}^{-}) \rbrace$ is an orbit set in $Y_-$ such that $[\alpha^{+}]$ and $[\alpha^{-}]$ represent the same class in $H_1(\overline{X})$, define $H_2(\overline{X},\alpha^{+},\alpha_{-})$ to be the set of relative homology classes of $2$-chains in $\overline{X}$ such that
\[
\partial Z = \sum_i m_i^+ \{1\}\times\alpha_i^+ - \sum_j m_j^-
\{-1\}\times\alpha_j^-.
\]
Here, two $2$-chains are equivalent if and only if their difference is the boundary of a $3$-chain.

Let $\tau$ be a homotopy class of symplectic trivializations $\tau^+_i$ of the restriction of $\xi_+ = \Ker(\lambda_+)$ to $\alpha_i^+$ and  $\tau^-_j$ of the restriction of $\xi_- = \Ker(\lambda_-)$ to $\alpha_j^-$.  Let $Z \in H_2(\overline{X},\alpha^+,\alpha_{-})$.  Define the ECH index, $I_{ECH}(Z)$ by the formula

\begin{equation}
\label{eqn:ECHIndex}
I_{ECH}(Z) \eqdef c_\tau(Z) + Q_\tau(Z) +
\mu_\tau(\alpha^+) - \mu_\tau(\alpha^-),
\end{equation}
where $c_{\tau}(Z)$ and $Q_{\tau}(Z)$ are respectively the {\em relative first Chern class\/} and the {\em relative intersection pairing\/} of $Z$ with respect to the trivialization $\tau$, as defined in \cite[\S 4.2]{hutchings_absolute}.  As explained in \cite[\S 4.2]{hutchings_absolute}, the ECH index does not depend on $\tau$. 

In the case where $(\overline{X},\omega)=(\R \times Y,d(e^s\lambda))$, the ECH index induces a relative $\Z/p$ grading on $ECH_*(Y,\lambda,\Gamma)$.  This is explained (for example) in \cite[\S 2.8]{hutchings_absolute}.                 

\section{Seiberg-Witten Floer cohomology}
\label{sec:sw}

We now review those aspects of Seiberg-Witten Floer cohomology that are relevant to the proofs of our main theorems.  For more details, see \cite{km}. 

\subsection{Basic terminology}
\label{sec:spinc}

Let $Y$ be a closed oriented Riemannian $3$-manifold.  A {\em spin$^c$} structure on $Y$ is a unitary rank-$2$ complex vector bundle $\mathbb{S} \to Y$ with a {\em Clifford multiplication},
\[\rho: TY \to \op{Hom}(\mathbb{S},\mathbb{S}).\]
The Clifford multiplication is required to identify $TY$ isometrically with the subbundle of traceless skew-adjoint endomorphisms equipped with the inner product $(a,b) \to \frac{1}{2}(a^*b)$.  It is also required to respect orientation, by which we mean that if $e_i$ is an oriented frame then $\rho(e_1)\rho(e_2)\rho(e_3)=1$.  Spin$^c$ structures exist over any closed oriented Riemannian $3$-manifold and the set of isomorphism classes of spin$^c$ structures is an affine space over $H^2(Y,\Z)$.  A {\em spinor\/} is a smooth section of $\mathbb{S}$.  A unitary connection $\mathbb{A}$ on $\mathbb{S}$ is called {\em spin$^c$} if parallel transport via $\mathbb{A}$ is compatible with the Clifford multiplication.  The set of spin$^c$ connections is an affine space over the space of imaginary valued $1$-forms. Associated to a spin$^c$ structure is the {\em determinant} line bundle $\op{det}(\mathbb{S})$.  This is the line bundle $\Lambda^2\mathbb{S}$.  If $\mathbb{A}$ is a spin$^c$ connection, we denote by $\mathbb{A}^t$ the induced connection on $\Lambda^2\mathbb{S}$.  A spin$^c$ connection is equivalent to a Hermitian connection on $\Lambda^2\mathbb{S}$.  
Given a spin$^c$ connection $\mathbb{A}$, define the {\em Dirac operator} $D_{\mathbb{A}}$ to be the composition
\[
\Gamma(Y,{\mathbb S}) \stackrel{\nabla_A}{\longrightarrow}
\Gamma(Y,T^*X \tensor {\mathbb S}) \stackrel{\rho}{\longrightarrow}
\Gamma(Y,{\mathbb S}).
\]
Here, the Clifford multiplication $\rho$ by $1$-forms is defined by the isomorphism between vector fields and $1$-forms induced by the metric. 

Over a closed oriented Riemannian $4$-manifold $X$, a spin$^c$ structure $\mathfrak{s}_X$ is again a unitary complex vector bundle $\mathbb{S}$, this time of rank $4$, together with a Clifford multiplication $\rho: TY \to \op{Hom}(\mathbb{S},\mathbb{S})$.  The requirements for $\rho$ to be a Clifford multiplication are similar to the requirements for the three-manifold case.  Spin$^c$ structures also exist over any $4$-manifold, and the set of isomorphism classes of spin$^c$ structures is again an affine space over $H^2(X,\Z)$.  This is all explained in \cite[\S 1.1]{km}.  Clifford multiplication extends to $k$-forms by the rule
\[
\rho(\alpha \wedge \beta)=\frac{1}{2}(\rho(\alpha)\rho(\beta)+(-1)^{\op{deg}(\alpha)\op{deg}(\beta)}\rho(\beta)\rho(\alpha)),
\]
and over a $4$-manifold Clifford multiplication by the volume form induces an important decomposition of $\mathbb{S}$ into two orthogonal rank-2 complex vector bundles, $S^+$ and $S^-$, where $S^+$ is defined to be the $-1$ eigenspace of Clifford multiplication by the volume form.  In the $4$-dimensional case, a spinor is again defined to be a section of $\mathbb{S}$, and a spin$^c$ connection is again defined by requiring that Clifford multiplication be parallel.  The connection on $\Lambda^2 S^+$ induced by a spin$^c$ connection $\mathbb{A}$ is denoted by $\mathbb{A}^t$.  As in the three-dimensional case, the space of spin$^c$ connections on $\mathfrak{s}_X$ is an affine space over $iT^*X$.  

The definition of the Dirac operator $D_{\mathbb{A}}$ for a spin$^c$ structure over a $4$-manifold is completely analogous to the definition in the three-dimensional case.  Over a $4$-manifold, the Dirac operator interchanges sections of $S^+$ and $S^-$ and hence we have a decomposition $D_{\mathbb{A}} = D_{\mathbb{A}^+} + D_{\mathbb{A}^-}$ where
\[D_{\mathbb{A}^+}:\Gamma(S^+) \to \Gamma(S^-),
\]
and
\[D_{\mathbb{A}^-}:\Gamma(S^-) \to \Gamma(S^+).\]                  


In dimensions three or four, an automorphism of a spin$^c$ structure $(\mathbb{S},\rho)$ is a bundle isomorphism of $\mathbb{S}$ that is compatible with $\rho$.  This is the same as a map from the underlying manifold into $S^1$.  We call the set of maps from the underlying manifold to $S^1$ the {\em gauge group\/} and we call elements of this group {\em gauge transformations\/}.  If $M$ is a $3$-manifold or a $4$-manifold and $\mathfrak{s}$ is a spin$^c$ structure over $M$,  denote by $\mathcal{C}(Y,\mathfrak{s})$ the space of pairs $(\mathbb{A},\Psi)$ such that $\mathbb{A}$ is a spin$^c$ connection and $\Psi$ is a spinor.  We call such a pair a {\em configuration\/} and call $\mathcal{C}$ the {\em configuration space\/}.  The gauge group acts on $\mathcal{C}$ by
\[ g \cdot (\mathbb{A},\Psi) \eqdef (\mathbb{A}-2g^{-1}dg,g\Psi).
\]   
           
\subsection{The three-dimensional Seiberg-Witten equations}
\label{sec:SWequations}

We will now introduce the three-dimensional Seiberg-Witen equations.  Let $Y$ be a closed oriented Riemannian $3$-manifold with spin$^c$ structure $\mathfrak{s}=(\mathbb{S},\rho).$  Fix an exact $2$-form $\mu$ on $Y$.  The {\em three-dimensional Seiberg-Witten equations with perturbation} are the equations for a configuration $(\mathbb{A},\Psi)$  given by
\begin{equation}
\label{eqn:CriticalPointEquation}
\begin{split}
D_{\mathbb{A}}\Psi &= 0, \\
* F_{\mathbb{A}^t} &= \langle \rho(\cdot)\Psi,\Psi\rangle+i*\mu. 
\end{split}
\end{equation}
Here, $F_{\mathbb{A}^t}$ denotes the curvature of $\mathbb{A}^t$. 
Fix a reference spin$^c$ connection $\mathbb{A}_0$.  Solutions of \eqref{eqn:CriticalPointEquation} are equivalent to critical points of the {\em perturbed Chern-Simons-Dirac functional\/}. This is the map $\mathcal{F}:\mathcal{C}(Y,\mathfrak{s}) \to \R$ defined by
\begin{equation}
\label{eqn:CSDFunctional}
\mathcal{F}(\mathbb{A},\varphi)=-\frac{1}{8}\int_Y (\mathbb{A}^t - \mathbb{A}_0^t) \wedge (F_{\mathbb{A}^t} + F_{\mathbb{A}_0^t}-2i\mu) + \frac{1}{2}\int_Y \langle D_{\mathbb{A}} \varphi,\varphi \rangle d\op{vol}. 
\end{equation}

While the functional $\mathcal{F}$ is not in general gauge invariant, the gauge group acts on solutions to \eqref{eqn:CriticalPointEquation}.

\subsection{Floer homology}
\label{sec:SWcohomology}

We now briefly review the details of the construction of the Seiberg-Witten Floer cohomology groups, which are related to the formal Morse homology of the functional $\mathcal{F}$.  Call a solution to \eqref{eqn:CriticalPointEquation} {\em reducible\/} if $\Psi = 0$ and call it irreducible otherwise.    
The Seiberg-Witten Floer cohomology chain complex $\widehat{CM}^*(Y,\mathfrak{s})$ can be decomposed into submodules
\[
\widehat{CM}^*(Y,\mathfrak{s}) =
\widehat{CM}^*_{irr}(Y,\mathfrak{s}) \oplus \widehat{CM}^*_{red}(Y,\mathfrak{s}),
\] 
where $\widehat{CM}^*_{irr}$ is the free $\Z/2$-module generated by gauge equivalence classes of irreducible solutions to \eqref{eqn:CriticalPointEquation} after choosing $\mu$ generically so that these solutions are cut out transversely, and $\widehat{CM}^*_{red}$ is another term involving the reducible solutions.  Only the irreducible component of this chain complex is relevant to the construction of the map $\mathcal{T}$ from \eqref{eqn:MainIsomorphism}, so we will not review the definition of $\widehat{CM}^*_{red}$ here.    

The part of the chain complex differential $\partial$ mapping the irreducible component to itself counts gauge equivalence classes of smooth one-parameter families of pairs $(\mathbb{A}(s),\Psi(s))$ that solve the equations
\begin{equation}
\label{eqn:instantoneqn}
\begin{split}
\frac{\partial}{\partial s} \Psi(s) & = -D_{\mathbb{A}(s)}\Psi(s), \\
\frac{\partial}{\partial s} \mathbb{A}(s) &= -* F_{\mathbb{A}(s)}+\langle cl(\cdot)\Psi,\Psi\rangle + i*\mu, \\
\lim_{s \to \pm\infty} (\mathbb{A}(s),\Psi(s)) &= (\mathbb{A}_{\pm},\Psi_{\pm}),
\end{split}
\end{equation}
where $(\mathbb{A}_{\pm},\Psi_{\pm})$ are solutions to \eqref{eqn:CriticalPointEquation}.  These are equations for the downward gradient flow of the functional \eqref{eqn:CSDFunctional} with respect to the metric on $\mathcal{C}$ induced by the Hermitian inner product on $\mathbb{S}$ and $1/4$ of the $L^2$ inner product on $iT^*Y$.  Solutions to \eqref{eqn:instantoneqn} are called {\em instantons}.  If $\mathfrak{c}_{\pm}$ are two irreducible solutions to \eqref{eqn:CriticalPointEquation}, then the coefficient of $\mathfrak{c}_-$ in the differential of $\mathfrak{c}_+$ is a signed count of gauge equivalence classes of ``index one" instantons from $\mathfrak{c}_-$ to $\mathfrak{c}_+$, modulo translation in the $s$ coordinate, after making ``abstract perturbations" to \eqref{eqn:CriticalPointEquation} and \eqref{eqn:instantoneqn} to obtain transversality of the relevant moduli spaces.  

``Abstract perturbations" are described in \cite[Ch. 11]{km} and play little role in the proof of Theorem~\ref{thm:MainTheorem}.  The ``index" is the local expected dimension of the moduli space of instantons modulo gauge equivalence.  The index induces a relative $\Z/p$ grading on the chain complex such that the differential increases the grading by $1$, see \cite[\S 2.1]{hutchings_arnold_chord}.  Here, $p$ is equal to the divisibility of $c_1(\mathfrak{s})$ in $H^2(Y,\Z)$ mod torstion. 

\subsection{The absolute grading of a critical point}
\label{sec:CanonicalGrading}

As is the case for embedded contact homology, the relative grading for $\widehat{HM}^{-*}(Y,\mathfrak{s})$ can be refined to an absolute grading.  To explain Kronheimer and Mrowka's construction, we need to introduce the {\em four-dimensional Seiberg-Witten equations}.  If $X$ is any (possibly non-compact) spin$^c$ $4$-manifold, the four-dimensional Seiberg-Witten equations (with perturbation) on $X$ for a configuration $(\mathbb{A},\Psi)$ is the system
\begin{equation}
\label{eqn:4DSW}
\begin{split}
\frac{1}{2}\rho(F^+_{\mathbb{A}^t})+\mathfrak{p}(\mathbb{A},\Psi)-(\Psi\Psi^*)_0 &= 0 \\
D_\mathbb{A}^+ \Psi &= 0.
\end{split}
\end{equation}
Here, $F^+_{\mathbb{A}^t}$ denotes the self-dual part of the curvature $2$-form, $(\Psi\Psi^*)_0$ denotes the traceless component of $\Psi\Psi^*$, and $\mathfrak{p}(\mathbb{A},\psi)$ denotes a gauge invariant perturbation term, see \cite[\S 24.1]{km}.  When $X=\R \times Y$, the system \eqref{eqn:4DSW} is equivalent to the system \eqref{eqn:instantoneqn} for an appropriate spin$^c$ structure, see \cite[\S 4.3]{km}.  The action of the gauge group on $\mathcal{C}$ induces an action on solutions of \eqref{eqn:4DSW}.    
 
To prove Theorem~\ref{thm:MainTheorem}, we only need to know the definition of the absolute grading for irreducible solutions to \eqref{eqn:CriticalPointEquation} that are nondegenerate i.e. cut out transversely (see \cite[Def. 12.1.1]{km} for the precise definition).  So let $\mathfrak{c}$ be such a solution and let $X$ be any compact connected oriented Riemannian $4$-manifold with oriented boundary $Y$ extending the spin$^c$ structure $\mathfrak{s}$ via a spin$^c$ structure $\mathfrak{s}_X$.  Assume that the Riemannian metric on $X$ is such that $X$ contains an isometric copy of $I \times Y$ for some interval $I = (-C,0],$ with $\partial X$ identified with $\lbrace 0 \rbrace \times Y$.  We can therefore {\em attach a cylindrical end to $X$} i.e. glue in a copy of the cylinder $[0,\infty) \times Y$ to $X$ to get a non-compact $4$-manifold $\overline{X}$ with spin$^c$ structure $\mathfrak{s}_{\overline{X}}$ extending the spin$^c$ structure on $\overline{X}$ via a translation invariant spin$^c$ structure on the end.
 
Denote the moduli space of gauge equivalence classes of configurations for the spin$^c$ structure $\mathfrak{s}_{\overline{X}}$ that are asymptotic (as in \cite[\S 13.1]{km}) to $\mathfrak{c}$ on the cylindrical end of $\overline{X}$ by $\mathcal{B}(\overline{X},\mathfrak{s}_{\overline{X}},\mathfrak{c})$ and denote the gauge equivalence classes of solutions to \eqref{eqn:4DSW} that are asymptotic to $\mathfrak{c}$ on the cylindrical end of $\overline{X}$ by $M(\overline{X},\mathfrak{s}_{\overline{X}},\mathfrak{c})$.  Here, the perturbation term to \eqref{eqn:4DSW} is constructed from the perturbation to \eqref{eqn:CriticalPointEquation}, see \cite[\S 24.1]{km}.  Denote by $\mathcal{B}(\overline{X},\mathfrak{c})$ and by $M(\overline{X},\mathfrak{c})$ the union of $\mathcal{B}(\overline{X},\mathfrak{s}_{\overline{X}},\mathfrak{c})$ and $M(\overline{X},\mathfrak{s}_{\overline{X}},\mathfrak{c})$ respectively over all spin$^c$ structures $s_{\overline{X}}$ on $\overline{X}$ extending $\mathfrak{s}$.

In general, the space $M(\overline{X},\mathfrak{c})$ can contain multiple connected components.  These are parametrized by $\pi_0(\mathcal{B}(\overline{X},\mathfrak{c}))$, which is an affine space over $H^2(X,\partial X,\Z)$. Let $z$ be an element of $\pi_0(\mathcal{B}(\overline{X},\mathfrak{c}))$.  Following \cite[Defn. 24.4.5]{km}, we now define an integer $gr_z(X,\mathfrak{c})$ which is the expected dimension of the component of $M(\overline{X},\mathfrak{c})$ corresponding to $z$.  If $(\mathbb{A},\Psi)$ is any element of $\mathcal{B}(\overline{X},\mathfrak{c})$, define the operator
\[
D^{\overline{X}}_{\mathbb{A},\Psi}: L^2_1(i T^* \overline{X})\oplus L^2_1(S^+)  \rightarrow  L^2(i\mathbb{R}) \oplus L^2(isu(S^{+}))\oplus L^2(S^-)
\]

by          
\begin{equation}\label{eq:operator}
D^{\overline{X}}_{\mathbb{A}, \Psi}(a,\varphi) =  (-d^{*}a+iIm(\Psi^{*}\varphi),\frac{1}{2}\rho(d^+a)-(\Psi \varphi^{*}+\varphi\Psi^{*})_0,D_{\mathbb{A}}^{+}\varphi+\rho(a)\Psi), 
\end{equation}
where $L^2_1(i T^*\overline{X}),L^2_1(S^+), L^2(i\mathbb{R}), L^2(isu(S^{+}))$, and $ L^2(S^-)$ denote Sobolev completions of the space of compactly supported smooth sections of these bundles over $\overline{X}$, see \cite[\S 13]{km}, and $d^+a$ denotes the self-dual component of $da$.  This is the linearization of the unperturbed $4$-dimensional Seiberg-Witten equations with a gauge fixing term.  As explained in \cite[\S 3.d]{taubes1} and \cite[Lem. 2.4]{taubes3}, when $\mathfrak{c}$ is irreducible and nondegenerate the operator $D^{\overline{X}}_{\mathbb{A},\Psi}$ is Fredholm.  The integer $gr_z(X,\mathfrak{c})$ is by definition the index of $D^{\overline{X}}_{\mathbb{A},\Psi}$ for $(\mathbb{A},\Psi)$ a lift of the gauge equivalence class of an element in the component of $\mathcal{B}(\overline{X},\mathfrak{c})$ corresponding to $z$.  As explained in \cite[\S 24]{km}, $gr_z(X,\mathfrak{c})$ can be defined for reducible solutions as well.  We call $gr_z(X,\mathfrak{c})$ the {\em Seiberg-Witten index}.    

If $\varphi_0$ is any section of $\mathbb{S}^+|_{\partial X}$, denote by $e(\mathbb{S}^+,\varphi_0) \in H^4(X,\partial X;\Z)$ the {\em relative Euler class} of $\mathbb{S}^+$ relative to $\varphi_0$.  To define the absolute grading, choose a nowhere-zero section $\varphi_0$ of $\mathbb{S}^{+}|_{\partial X}$ such that $e(\mathbb{S}^{+},\varphi_0)[X,\partial X]=gr_z(X;\mathfrak{c})$.  The pair $(\mathbb{S}^{+}|_{\partial X},\varphi_0)$ is a spin$^c$ structure on $Y$ equipped with a non-zero section, so we can apply the following basic lemma \cite[Lem. 28.1.1]{km}:

\begin{lemma}
\label{lem:bijection}
On an oriented Riemannian $3$-manifold $Y$, there is a one-to-one correspondence between oriented $2$-plane fields $\xi$ and isomorphism classes of pairs $(\mathfrak{s},\varphi)$ consisting of a spin$^c$ structure and a unit-length spinor $\varphi$.
\end{lemma} 

By \cite[Prop. 28.2.2]{km}, the isomorphism class of $(\mathbb{S},\varphi_0)$ depends only on $Y, \mathfrak{s},$ and $\mathfrak{c}$, and so the bijection of Lemma~\ref{lem:bijection} induces a well-defined grading by homotopy classes of oriented $2$-plane fields, which we denote by $I_{SW}$.  This refines the relative grading on $\widehat{HM}^{-*}(Y,\mathfrak{s})$, see \cite[\S 28]{km}.  The absolute grading can be defined for reducible critical points as well, see \cite[\S 28]{km}.

\begin{remark}
\label{rmk:signs}
Our sign convention (as explained in \S\ref{sec:ECHAbsoluteGrading}) for the $\Z$-action on the set of homotopy classes of $2$-plane fields is opposite the sign convention in \cite[\S 28]{km}.  This is because the grading defined by Kronheimer and Mrowka refines the relative grading on $\widehat{HM}^*$, while our grading refines the relative grading on $\widehat{HM}^{-*}$.
\end{remark}

\section{Taubes' isomorphism}

This section very briefly summarizes Taubes' isomorphism between embedded contact homology and Seiberg-Witten Floer cohomology.  For more details, see \cite{echswf}.  

\subsection{Taubes' equations}
\label{sec:TaubesPerturbation}

Let $(Y,\lambda)$ be a contact manifold.  A choice of admissible almost complex structure $J$ induces a metric $g$ on $Y$ by requiring that the Reeb vector field $R$ has length $1$, is orthogonal to the contact planes $\xi$, and
\begin{equation}
\label{eqn:Jg}
g(v,w) = \frac{1}{2}d\lambda(v,Jw), \quad\quad v,w\in\xi_y.
\end{equation}
Let $\mathbb{S}$ be the spin bundle for the spin$^c$ structure $\mathfrak{s}_{\xi} + \op{PD}(\Gamma)$.  Clifford multiplication by $\lambda$ gives a decomposition
\[
\mathbb{S}=E \oplus (E \otimes \xi),
\]
where $E$ and $E \otimes \xi$ are, respectively, the $+i$ and $-i$ eigenspaces of Clifford multiplication by $\lambda$.  Here $\xi$ is regarded as a complex line bundle.

Connections on $\op{det} \mathbb{S}$ can therefore be written as $\mathbb{A}_0 + 2\mathbb{A}$ where $\mathbb{A}_0$ is a certain fixed connection on $\xi$, as reviewed in ~\cite[\S 2.a]{tw}, and $\mathbb{A}$ is a connection on $E$.  We can therefore regard a connection on $E$ as a connection on $\op{det} \mathbb{S}$.  With this in mind, consider the system of equations for a connection $\mathbb{A}$ on $E$ and a spinor $\psi$ given by
\begin{equation}
\label{eqn:taubessw}
\begin{split}
*F_{\mathbb{A}}&=r(\langle \rho(\cdot)\psi,\psi\rangle-i\lambda)+i(* d\mu+\bar{\omega})\\
D_{\mathbb{A}}\psi&= 0.
\end{split}
\end{equation}
Here, $\bar{\omega}$ denotes the harmonic $1$-form such that $*\frac{\bar{\omega}}{\pi}$  represents the image of $c_1(\xi)$ in $H^2(Y;\R)$, $r$ is a positive real number, and $\mu$ is a suitably generic coclosed $1$-form that is $L^2$-orthogonal to the space of harmonic $1$-forms and that has ``P-norm" less than $1$.  The P-norm controls the derivatives of $\mu$ to all orders, see  \cite[\S2.2]{hutchings_arnold_chord}.  This is a a special case of \eqref{eqn:CriticalPointEquation} where we have also rescaled the spinor by $\sqrt{r}$.  
    
If $\mu$ is generic, then all of the irreducible solutions to \eqref{eqn:taubessw} are nondegenerate.  One can also make additional small perturbations to the equations so that the moduli spaces needed to define the chain complex differential are all cut out transversely.  Moreover, in any fixed grading, if $r$ is sufficiently large, these additional perturbations can be chosen such that only irreducible solutions to this perturbed version of \eqref{eqn:taubessw} contribute to the Seiberg-Witten cohomology chain complex in that grading, see \cite[Prop. 3.5]{tw}.  By \cite[\S 2.1]{hutchings_arnold_chord}, these perturbations can be chosen to vanish to any given order on the irreducible solutions to \eqref{eqn:taubessw}, so that the irreducible solutions to \eqref{eqn:taubessw} and the solutions to this perturbed version of \eqref{eqn:taubessw} are the same.         

\subsection{Taubes' proof}
\label{sec:taubesproof}

The basic idea behind the isomorphism \eqref{eqn:MainIsomorphism} is that as $r$ gets very large, the zero set of the $E$ component of the spinor for solutions of \eqref{eqn:taubessw} converges (as a current) to an ECH chain complex generator, and the symplectic action of this chain complex generator is very close to $2\pi$ times the ``energy" of the solution.  

To state this precisely, recall that if $\alpha=\{(\alpha_i,m_i)\}$ is a generator of the ECH chain complex, the {\em symplectic action\/} of $\alpha$ is the number
\[
\mc{A}(\alpha) \eqdef \sum_im_i \int_{\alpha_i}\lambda.
\]
Because of the conditions on $J$, the ECH chain complex differential decreases the symplectic action.  Hence, for any real number $L$, we can define {\em filtered ECH\/}, $ECH^{L}(Y,\lambda,\Gamma)$, to be the homology of the subcomplex of the ECH chain complex spanned by generators with action strictly less than $L$.
 
Given a configuration $(\mathbb{A},\Psi)$, define the {\em energy\/}
\begin{equation}
\label{eqn:energy}
E(\mathbb{A}) \eqdef i\int_Y \lambda\wedge F_A,
\end{equation}
and define $\widehat{CM}^{*}_L(Y,\mathfrak{s},\lambda,r)$ to be the submodule of $\widehat{CM}^{*}_{irr}$ generated by irreducible solutions $(\mathbb{A},\Psi)$ to \eqref{eqn:CriticalPointEquation} (perturbed as in \S\ref{sec:TaubesPerturbation}) with energy less than $2\pi L$.  If $r$ is sufficiently large, and $\lambda$ has no orbit set of action exactly $L$, then one can show \cite[Lem. 2.3]{hutchings_arnold_chord} that all of the solutions to \eqref{eqn:taubessw} with energy less than $2 \pi L$ are irreducible and the chain complex differential for $\widehat{CM}^*(Y,\mathfrak{s},\lambda,r)$ maps $\widehat{CM}^*_L(Y,\mathfrak{s},\lambda,r)$ to itself. 

The key fact (\cite[Prop. 3.1]{hutchings_arnold_chord}) needed for the proof of \eqref{eqn:MainIsomorphism} is that if $r$ is sufficiently large and $(\lambda,J)$ is ``$L$-flat", then for any $\Gamma \in H_1(Y)$, there is a canonical bijection between the set of generators of $\widehat{CM}_L^{-*}(Y,\mathfrak{s}_{\xi}+\op{PD}(\Gamma);\lambda,r)$ and the set of admissible orbit sets in the homology class $\Gamma$ of length less than $L$.  This induces an isomorphism of relatively graded chain complexes
\begin{equation}
\label{eqn:CanonicalIso}
ECC_*^L(Y,\lambda,\Gamma) \stackrel{\simeq}{\longrightarrow} \widehat{CM}_L^{-*}(Y,\mathfrak{s}_{\xi}+\op{PD}(\Gamma);\lambda,r),
\end{equation}
which, as explained in \cite[\S 3]{hutchings_arnold_chord}, induces the isomorphism $\mathcal{T}$ between $ECH(Y,\lambda,\Gamma)$ and $\widehat{HM}^{-*}(Y,\mathfrak{s}_{\xi+\op{PD}(\Gamma)})$.  Roughly speaking, the bijection between chain complex generators is given by constructing an approximate solution to \eqref{eqn:taubessw} for large $r$ from an ECH chain complex generator by using the ``vortex equations", see \cite{taubes1}, and then using perturbation theory to get an actual solution to \eqref{eqn:taubessw}.  

The $L$-flat condition is a condition on the form of $\lambda$ and $J$ in tubular neighborhoods of those Reeb orbits with action less than $L$.  In the case where $(\lambda,J)$ is not $L$-flat, one can take an {\em L-flat approximation} of $\lambda$: a pair $(\lambda,J)$ of nondegenerate contact form and admissible almost complex structure can always be approximated by an $L$-flat pair $(\lambda_1,J_1)$ without changing the Reeb orbits or the lengths of the orbits with action less than $L$, and this identification induces an isomorphism of chain complexes
\begin{equation}
\label{eqn:lflat}
ECC_*^L(Y,\lambda,\Gamma;J) \stackrel{\simeq}{\longrightarrow} ECC_*^L(Y,\lambda_1,\Gamma;J_1).
\end{equation}     
This is all explained in \cite[\S 3]{hutchings_arnold_chord}.
\section{Proof of theorems}

\subsection{The Seiberg-Witten index in a symplectic cobordism}
\label{sec:swindex}

To prove Theorem~\ref{thm:MainTheorem}, we will first prove another theorem relating the expected dimension of any component of the Seiberg-Witten moduli space over a symplectic cobordism to the ECH index of a corresponding relative homology class.  

To be specific, let $(X,\omega)$ be a connected symplectic cobordism from $(Y_1,\lambda_1)$ to $(Y_2,\lambda_2)$ as in \S\ref{sec:ECHIndex}, and denote by $\overline{X}$ the symplectic completion of $X$.  Let $J$ be an admissible almost complex structure on $\overline{X}$, and let $g$ be the Riemannian metric induced by $\omega$ and $J$.  Let $\alpha_1$ be an orbit set on $Y_1$ and let $\alpha_2$ be an orbit set on $Y_2$. Assume that the contact forms $\lambda_1$ and $\lambda_2$ are ``$L$-flat", where $L$ is some constant greater than the symplectic action of either $\alpha_1$ or $\alpha_2$.  Recall that the canonical isomorphism \eqref{eqn:CanonicalIso} is induced from a canonical bijection between the set of generators of $\widehat{CM}_L^{-*}(Y,\mathfrak{s}_{\xi}+\op{PD}(\Gamma);\lambda,r)$ and the set of admissible orbit sets in the homology class $\Gamma$ of length less than $L$, and denote by $c_{\alpha_1}$ and $c_{\alpha_2}$ the elements corresponding to $\alpha_1$ and $\alpha_2$ respectively under this bijection.  By \cite[\S 2.a]{taubes3}, if $r$ is sufficiently large, then $c_{\alpha_1}$ and $c_{\alpha_2}$ are both nondegenerate and belong to the irreducible component of the chain complex $\widehat{CM}^*.$  

Let $\mathfrak{s}_{Y_1}$ and $\mathfrak{s}_{Y_2}$ denote the spin$^c$ structures on $Y_1$ and $Y_2$ corresponding to $c_{\alpha_1}$ and $c_{\alpha_2}$ respectively.  Then $c_{\alpha_1}, c_{\alpha_2},\mathfrak{s}_{Y_1}$, and $\mathfrak{s}_{Y_2}$ induce a spin$^c$ structure $\mathfrak{s}_Y$ and configuration $\mathfrak{c}$ on $Y=Y_1 \cup -Y_2$.
Recall the space $\mathcal{B}(\overline{X},\mathfrak{c})$ from \S\ref{sec:CanonicalGrading}, and let $(\mathbb{A},\Psi)$ be an element of $\mathcal{B}(\overline{X},\mathfrak{c})$.  The configuration $(\mathbb{A},\Psi)$ determines a spin$^c$ structure $\mathfrak{s}_{\mathbb{A},\Psi}$ over $\overline{X}$.  As before, denote by $S^+$ the $-1$ eigenspace of Clifford multiplication by the volume form on the spin$^c$ structure $\mathfrak{s}_{\mathbb{A},\Psi}$.  Since $\overline{X}$ is symplectic, we can write 
\[S^{+} = E \oplus (E \otimes K^{-1}),\]
where $K^{-1}$ denotes the inverse of the canonical bundle and $E$ and $E \otimes K^{-1}$ are, respectively, the $-2i$ and $+2i$ eigenspaces of Clifford multiplication by the symplectic form.  This is reviewed, for example, in \cite[\S 4.2]{survey}.  We can then write the spinor 
\[\Psi=(\alpha,\beta)\]
according to this decomposition, where $(\mathbb{A},\Psi)$ now denotes a specific lift of its gauge equivalence class. Assume that $(\mathbb{A},\Psi)$ is such that $\alpha$ intersects the zero section transversally.  Hence, $\alpha^{-1}(0)$ is an embedded (real) surface.  Denote this surface by $C_{\mathbb{A},\Psi}$.

Recall that, as reviewed in \S\ref{sec:taubesproof}, as $r$ gets very large, the zero sets of $c_{\alpha_1}$ and $c_{\alpha_2}$ converge as currents to $\alpha_1$ and $\alpha_2$, respectively.  By taking orientation-preserving diffeomorphisms $[0,\infty) \simeq [0,1-\epsilon)$ and $(-\infty,0] \simeq (-1+\epsilon,0]$ to identify
\[
\overline{X} \simeq ((-1+\epsilon,0]\times Y_-) \cup_{Y_-} X \cup_{Y_+} (([0,1-\epsilon)\times
Y_+).
\]
and composing the closure of the image of $C_{\mathbb{A},\Psi}$ in the latter with cobordisms to the Reeb orbits in the orbit sets $\alpha_1$ and $\alpha_2$, the curve $C_{\mathbb{A},\Psi}$ defines an element $Z_{\mathbb{A},\Psi} \in H_2(\overline{X},\alpha_1,\alpha_2)$.  We can relate $I_{ECH}(Z_{\mathbb{A},\Psi})$ to the expected dimension of the corresponding Seiberg-Witten moduli space, as the following theorem shows:

\begin{theorem}
\label{thm:swindex}
Let $z \in \pi_0(\mathcal{B}(\overline{X},\mathfrak{c}))$ and represent $z$ by a configuration $(\mathbb{A},\Psi)$ over $\overline{X}$.  The integer $gr_z(X,\mathfrak{c})$ is equal to $I_{ECH}(Z_{\mathbb{A},\Psi})$.  
\end{theorem}

\begin{proof}

Our method of proof closely tracks the argument due to Taubes in \cite[\S 2.b]{taubes3}.  The basic approach is to change the triple $(\overline{X},J,\omega)$ into a new triple $(\tilde{X},\tilde{J},\tilde{\omega})$ (with $\tilde{\omega}$ nondegenerate but not necessarily symplectic) in which the homology class $Z_{\mathbb{A},\Psi}$ induces a homology class $\tilde{Z}_{\mathbb{A},\Psi}$ with a $\tilde{J}$-holomorphic representative with ends of a particularly nice form.  An argument due to Taubes then generalizes without difficulty to allow us to compute the ECH index of $\tilde{Z}_{\mathbb{A},\Psi}$, and it is straightforward to relate this index to the ECH index of $Z_{\mathbb{A},\Psi}$.  The details are given in three steps.       

{\em Step 1. \/}  First, choose a representative $C_z$ of the homology class of $Z_{\mathbb{A},\Psi}$ with no compact components and with ends of the special form described in \cite[\S 2.b.1]{taubes3}.  In particular, the requirements from \cite[\S 2.b.1]{taubes3} imply that the ends of $C_z$ are asymptotic to the orbit set $\alpha_1$ at $+\infty$, asymptotic to the orbit set $\alpha_2$ at $-\infty$, and converge exponentially fast.  We can then find a pair $(\tilde{J},\tilde{\omega})$, where $\tilde{J}$ is an almost complex structure on a neighborhood of $C_z$ such that $C_z$ is $\tilde{J}$-holomorphic and $\tilde{\omega}$ is a (not necessarily closed) self-dual $2$-form on $\overline{X}$ with transverse zero locus whose restriction to $C_z$ is compatible with $\tilde{J}$.  We can assume that the pair $(\tilde{J},\tilde{\omega})$ satisfies the analogues of the additional technical conditions required in \cite[\S 2.b.2]{taubes3}.  Note that these conditions force $\tilde{\omega}$ to converge exponentially fast to $ds \wedge \lambda_{\pm} + *\lambda_{\pm}$ as the norm of the $\R$-coordinate $s$ on each cylindrical end tends to infinity.    
 
Denote the zero locus of the $2$-form $\tilde{\omega}$ by $B$.  Note that $B$ consists of a finite number of disjoint embedded circles which are also disjoint from $C_z$.  Let $T$ denote a tubular neighborhood of $B$ that is disjoint from $C_z$.  We can assume that $B$ has the special description given in \cite[\S 2.b.2]{taubes3}, so that we can copy the argument in \cite[\S 2.b.3]{taubes3} to modify the manifold $\overline{X}$ and the metric on $\overline{X}$ in $T$ to obtain a new Riemannian manifold $\tilde{X}$, obtained by surgery along $T$, such that $\tilde{\omega}$ extends to a nonvanishing self-dual $2$-form on $\tilde{X}$ (which we also denote by $\tilde{\omega}$) and such that the spin$^c$ structure on $\overline{X} - T$ extends to a spin$^c$ structure on $\tilde{X}$.

Now denote the canonical bundle on $(\tilde{X},\tilde{\omega})$ by $\tilde{K}^{-1}$.  The self-dual part of the spinor bundle for the spin$^c$ structure on $\tilde{X}$ splits as $E \oplus E\tilde{K}^{-1}$ with respect to Clifford multiplication by $\tilde{\omega}$.  It will be important to understand the relationship between $\tilde{K}$ and $K$ explicitly.  To do this, recall that there is a canonical spin$^c$ structure on $\overline{X}$ with self-dual component $\C \oplus \C K^{-1}.$  Denote the $+i |\tilde{\omega}|$ eigenspace of Clifford multiplication by $\tilde{\omega}$ on the self-dual component of this spin$^c$ stucture over $\overline{X}\setminus B$ by $L$.  Then, as explained in \cite[\S4.b]{taubes3}, we have 
\[\tilde{K}=L^2K.\]  This description for $L$ ensures that we can choose $t_1, t_2$ such that $Y_1 \times \lbrace t_1 \rbrace$ and $Y_2 \times \lbrace t_2 \rbrace$ are both in $\overline{X}-T$ and the restriction of $L$ to $Y_1 \times [t_1,\infty)$ and $Y_2 \times (-\infty,t_2]$ is canonically isomorphic to the trivial bundle.

{\em Step 2.\/}  We can now copy the construction from \cite[\S2.b.6]{taubes3} to construct a particular irreducible configuration $(\mathbb{A}_s,\Psi_s)$ for our spin$^c$ structure over $\tilde{X}$ with large $|s|$ limit gauge equivalent to $\mathfrak{c}$.  Let $k_L$ denote the relative first Chern class of $L$ evaluated on $C_z$, relative to the section $1$ on $Y_1 \times \lbrace t_1 \rbrace$ and $Y_2 \times \lbrace t_2 \rbrace$.  The significance of the configuration $(\mathbb{A}_s,\Psi_s)$ is given by the following proposition:

\begin{proposition}
\label{prop:keyprop}
The index of $D_{\mathbb{A}_s,\Psi_{s}}^{\tilde{X}}$ is equal to $I_{ECH}(Z_{\mathbb{A},\Psi})-2k_L$.
\end{proposition}
\begin{proof}
This is proved (in different notation) in \cite[\S 2c]{taubes3}.  In this section, Taubes is working over a manifold which arises via surgery on the symplectization of a contact $3$-manifold $Y$, but his argument also holds in the slightly greater generality we require, see Remark~\ref{rem:keyremark} below.  
\end{proof}

\begin{remark}
\label{rem:keyremark}
It is worth summarizing Taubes' argument from \cite[\S 2c]{taubes3}, since this is the key step in the proof of Theorem~\ref{thm:swindex}.  This will also clarify why his argument holds in the greater generality we are demanding.  

To motivate Taubes' argument, we need to review how Taubes in \cite{taubes2} constructs a Seiberg-Witten instanton with the appropriate asymptotics from a curve counted by the ECH chain complex differential.  Recall from \S\ref{sec:taubesproof} that the bijection between chain complex generators that induces the isomorphism \eqref{eqn:CanonicalIso} is given by using solutions to the vortex equations to construct approximate solutions to Taubes' perturbed Seiberg-Witten equations and then using perturbation theory.  To construct an instanton from an ECH index one $J$-holomorphic curve, Taubes again uses the vortex equations to construct an approximate solution and uses perturbation theory to produce an instanton.  

This approximate solution is essentially the configuration $(\mathbb{A}_s,\Psi_s)$.  To construct an instanton, Taubes considers a family of deformations of $(\mathbb{A}_s,\Psi_s)$ parametrized by a certain Banach space 
\[\mathcal{K} \hookrightarrow \ L^2_1(i T^* \overline{X})\oplus L^2_1(S^+),\] 
where the $\hookrightarrow$ means that the map is an injection (in fact, it can be made nearly isometric after putting the norm described in \cite[Equation 2.63]{taubes3} on $L^2_1(i T^* \overline{X})\oplus L^2_1(S^+)$).  The space $\mathcal{K}$ is also constructed using the vortex equations.  Taubes then shows that constructing an instanton by perturbing $(\mathbb{A}_s,\Psi_s)$ is equivalent to solving the projection of the relevant PDE onto another Banach space
\[\mathcal{L} \hookrightarrow L^2(i\mathbb{R}) \oplus L^2(isu(S^{+}))\oplus L^2(S^-),\] 
see \cite[\S 7]{taubes2}, which Taubes then solves by using the contraction mapping theorem.  The basic idea behind Taubes' method for the index computation in \cite[\S 2.c]{taubes3} is to decompose the operator $D_{\mathbb{A}_s,\Psi_{s}}^{\tilde{X}}$ to get an operator,
\[\Delta:\mathcal{K} \to \mathcal{L}.\]  
Taubes shows that the index of $D_{\mathbb{A}_s,\Psi_{s}}^{\tilde{X}}$ is equal to the index of $\Delta$, and the kernel and cokernel of the operator $\Delta$ can both be described explicitly, see \cite[\S 2.c.3]{taubes3}.  At any rate, for our purposes, the key point is that all the relevant analysis takes place local to the curve $C_z$, hence the generalization to a cobordism with cylindrical ends.
\end{remark}   


{\em Step 3.\/}  We now complete the proof by comparing the index of $D_{\mathbb{A}_s,\Psi_s}^{\tilde{X}}$ to the index of $D^{\overline{X}}_{\mathbb{A},\Psi}.$    

Denote the component of $\overline{X}$ bounded by $Y_1 \times \lbrace t_1 \rbrace$ and $Y_2 \times \lbrace t_2 \rbrace$ by $M$ and denote the corresponding component of $\tilde{X}$ by $\tilde{M}$.  Glue $M$ to $\tilde{M}$ (reversing the orientation on $\tilde{M}$) along their common boundary to obtain a closed spin$^c$ $4$-manifold $(S,\mathfrak{s}_S)$.  Let $(\mathbb{A}_S,\Psi_S)$ be a configuration on $(S,\mathfrak{s}_S)$. The additivity of $\op{gr}$ under gluing (e.g. as explained in \cite{km}) implies that

\begin{equation}
\label{eqn:addivity}
\op{ind}(D^{\overline{X}}_{\mathbb{A},\Psi})=\op{ind}(D^{\tilde{X}}_{\mathbb{A}_s,\Psi_s})+\op{ind}(D^S_{\mathbb{A}_S,\Psi_S}).
\end{equation}

It is a simple matter to compute the index of $\op{ind}(D^S_{\mathbb{A}_S,\Psi_S})$.  Indeed, by \cite[Thm. 1.4.1]{km}, we have
\begin{equation}
\label{eqn:preindex}
\op{ind}(D^S_{\mathbb{A}_S,\Psi_S})= \frac{1}{4}(c_1(S^+)^2[S]-2\chi(S)-3\sigma(S)),
\end{equation}
where $\sigma$ denotes the signature of $S$, and by \cite[Lem.\ 28.2.3]{km} we also know that 
\begin{equation}
\label{eqn:kmlemma}
(c_2(S^+)-\frac{1}{4}c_1(S^+)^2)[S]=-\frac{1}{4}(2\chi(S)+3\sigma(S)).
\end{equation}

Combining these two equations gives
\begin{equation}
\label{eqn:index}
\op{ind}(D^S_{\mathbb{A}_S,\Psi_S})=c_2(S^+)[S].
\end{equation}

We therefore have
\begin{equation}
\label{eqn:difference}
\begin{split}
\op{ind}(D^S_{\mathbb{A}_S,\Psi_S}) &= 2(c_1(E) \cup c_1(L))[M] \\
 &= 2 k_L.
\end{split}
\end{equation} 

The result now follows by combining Proposition~\ref{prop:keyprop}, \eqref{eqn:addivity},  
and \eqref{eqn:difference}.
\end{proof}

\subsection{A concave symplectic filling}

Our strategy for proving Theorem~\ref{thm:MainTheorem} is to apply Theorem~\ref{thm:swindex} to an appropriate cobordism.  To produce this cobordism, let $\Gamma \in H_1(Y)$ and fix an orbit set $\alpha \in ECC(Y,\lambda,\Gamma)$.  Recall from \cite[Thm. 2.5]{etnyre_legendrian_knot_type} that any smooth knot can be $C^0$ approximated by a Legendrian knot.  Thus, we can choose a Legendrian knot $\mathcal{K}$ which represents the class $\Gamma$.

Recall now the concept of Legendrian surgery.   This is reviewed, for example, in \cite{etnyre_symplectic_cobordisms}.  Recall also from \cite[\S 1.6]{cc1} that if $\mathcal{K}$ is a Legendrian knot in $(Y,\lambda)$, then one can perform a  Legendrian surgery along $\mathcal{K}$ to obtain another contact $3$-manifold $(Y',\lambda')$ such that there exists a symplectic cobordism from $(Y,\lambda)$ to $(Y',\lambda')$ obtained by attaching a $2$-handle along a tubular neighborhood of $\mathcal{K}$.  Recall that a {\em concave symplectic filling\/} of $(Y,\xi)$ is a symplectic cobordism from $(Y,\xi)$ to the empty set.  Concerning concave symplectic fillings, Etnyre and Honda prove \cite[Thm.\ 1.3]{etnyre_symplectic_cobordisms} that any contact $3$-manifold has infinitely many concave symplectic fillings.

Given an orbit set $\alpha$, we can therefore combine these results to define a manifold $X_{\alpha}$ by first performing Legendrian surgery on $Y$ along $\mathcal{K}$ to obtain another contact $3$-manifold and then composing the resulting symplectic cobordism with a concave symplectic filling.  In the next section, we will apply Theorem~\ref{thm:swindex} to $X_{\alpha}$.  

\subsection{Proof of main theorem}
To prove Theorem~\ref{thm:MainTheorem}, we will assume that the contact form is $L$-flat and show that the canonical bijection \eqref{eqn:CanonicalIso} preserves the absolute gradings.  This will prove the theorem for any contact form $\lambda$, since the isomorphism \eqref{eqn:lflat} preserves the absolute grading.  So, assume that the contact form is $L$-flat, let $\alpha \in ECC^L(Y,\lambda,\Gamma)$ be an orbit set, and denote by $\mathfrak{c}_{\alpha}$ the element corresponding to $\alpha$ under the canonical bijection between the set of generators of $\widehat{CM}_L^{-*}(Y,\mathfrak{s}_{\xi}+\op{PD}(\Gamma);\lambda,r)$ and the set of admissible orbit sets in the homology class $\Gamma$ of length less than $L$.

Recall from \S\ref{sec:ECHAbsoluteGrading} that the ECH absolute grading is given by 
\begin{equation}
\label{eqn:recalledECHabsolutegradingequation}
I_{ECH}(\alpha) \eqdef P_\tau(\mathcal{L}) - \sum_i w_{\tau_i}(\zeta_i) +
  \mu_\tau(\alpha),
\end{equation}
where $w_{\tau_i}(\zeta_i)$ is the writhe of a braid $\zeta_i$ around $\alpha_i$ with $m_i$ strands, $\mu_{\tau}(\alpha)$ is a certain sum of Conley-Zehnder index terms associated to $\alpha$, and $\mathcal{L}$ is the union of the $\zeta_i$.  To relate $I_{ECH}(\alpha)$ to $I_{SW}(\mathfrak{c}_{\alpha})$, begin by recalling the symplectic manifold $X_{\alpha}$ defined in the previous section. Let $\overline{X}_{\alpha}$ denote the manifold $X_{\alpha}$ with cylindrical ends attached.  Recall that the homotopy class of two plane fields $P_{\tau}(\mathcal{L})$ determines a spin$^c$ structure $\mathfrak{s}(P_{\tau}(\mathcal{L}))$.  By \cite[Thm 3.1(b)]{hutchings_absolute}, 
\[\mathfrak{s}(P_{\tau}(\mathcal{L}))=\mathfrak{s}_{\xi}+\op{PD}([\alpha]).\]
Remember that $[\alpha]$ vanishes in $H_1(X_{\alpha})$.  Since $\mathfrak{s}_{\xi}$ extends to a spin$^c$ structure on $\overline{X}_{\alpha},$ it follows that $\mathfrak{s}(P_{\tau}(\mathcal{L}))$ does as well.  

To simplify the notation, denote the ``plus" summand of the spin bundle for the extension of $\mathfrak{s}(P_{\tau}(\mathcal{L}))$ to $X_{\alpha}$ by $S^+_{\alpha}$ and denote $\mathfrak{s}(P_{\tau}(\mathcal{L}))$ by $\mathfrak{s}_{\alpha}$.  Recall from \S\ref{sec:CanonicalGrading} that $I_{SW}(\mathfrak{c}_{\alpha})$ is the homotopy class of two-plane fields corresponding to $(\mathfrak{s}_{\alpha},\varphi_0),$ where $\varphi_0$ is a section of $S^+_{\alpha}|_Y$ satisfying
\begin{equation}
\label{eqn:canonicaldefinition}	
e(S^+_{\alpha},\varphi_0)[X_{\alpha},\partial X_{\alpha}]=gr_z(X_{\alpha};\mathfrak{c}_{\alpha}),
\end{equation}
and $z$ is any element of $\pi_0(\mathcal{B}(\overline{X}_{\alpha},\mathfrak{c}_{\alpha}))$.  For $\varphi$ any section of $S^+_{\alpha}|_Y$, denote by $\tilde{e}(S_{\alpha}^+,\varphi) \in \Z$ the {\em relative Euler number} $e(S_{\alpha}^+,\varphi)[X_{\alpha},\partial X_{\alpha}].$  Recall that the set of homotopy classes of $2$-plane fields in a given spin$^c$ structure has a $\Z$-action.  This induces an action on the second component of isomorphism classes of pairs $(\mathfrak{s}_{\alpha},\varphi),$ where $\varphi$ is a nowhere zero section.  With respect to this $\Z$-action, the relative Euler number satisfies:
\begin{equation}
\label{eqn:basicproperty}
\tilde{e}((S^+_{\alpha},\varphi)+a)=\tilde{e}(S^+_{\alpha},\varphi)-a.
\end{equation}
In particular, it follows from \eqref{eqn:canonicaldefinition} and \eqref{eqn:basicproperty} that 
\begin{equation}
\label{eqn:eqn1}
I_{SW}(\mathfrak{c}_{\alpha})=(\mathfrak{s}_{\alpha},\varphi)+\tilde{e}(S^+_{\alpha},\varphi)-gr_z(X_{\alpha};\mathfrak{c}_{\alpha}),
\end{equation}
where $\varphi$ is any section.

To relate \eqref{eqn:eqn1} to \eqref{eqn:recalledECHabsolutegradingequation}, let $\varphi_{\mathcal{L}}$ be such that $(\mathfrak{s}_{\alpha},\varphi_{\mathcal{L}})=P_{\tau}(\mathcal{L})$.  Let $\Psi$ be a section of $S^+_{\alpha}$ extending $\varphi_{\mathcal{L}}$ and transverse to the zero section, and write $S^+_{\alpha}=E \oplus (E\otimes K^{-1})$ over $X_{\alpha}$.  Write $\Psi = (\gamma,\tilde{\gamma})$ with respect to this decomposition.  The zero set of $\gamma$ defines an embedded real surface in $X_{\alpha}$, which we will denote by $C_{\mathcal{L}}$.  Composing $C_{\mathcal{L}}$ with a cobordism to the Reeb orbits in $\alpha$ determines a homology class $Z_{\mathcal{L}} \in H_2(\overline{X},\emptyset,\alpha).$  We can now apply Theorem~\ref{thm:swindex} to choose $z \in \pi_0(\mathcal{B}(\overline{X}_{\alpha},\mathfrak{c}_{\alpha}))$ such that 
\begin{equation}
\label{eqn:keyequation}
I_{ECH}(Z_{\mathcal{L}})=gr_z(X_{\alpha},\mathfrak{c}_{\alpha}).
\end{equation}       

By \eqref{eqn:eqn1} and \eqref{eqn:keyequation}, we therefore have
\begin{equation}
\label{eqn:bigequation}
I_{SW}(\mathfrak{c}_{\alpha})=(\mathfrak{s}_{\alpha},\varphi_{\mathcal{L}})+\tilde{e}(S^+_{\alpha},\varphi_{\mathcal{L}})-I_{ECH}(Z_{\mathcal{L}}).
\end{equation}
   
By the definition of $\varphi_{\mathcal{L}}$, $P_{\tau}(\mathcal{L})=(\mathfrak{s}_{\alpha},\varphi_{\mathcal{L}})$.  To complete the proof, we therefore just need to show that
\begin{equation}
\label{eqn:laststep}
\tilde{e}(S^+_{\alpha},\varphi_{\mathcal{L}})=-\sum_i \omega_{\tau_i}(\zeta_i)+\mu_{\tau}(\alpha)+I_{ECH}(Z_{\mathcal{L}}).
\end{equation} 

This computation is easiest if we choose a particular representative of the isomorphism class of $(\mathfrak{s}_{\alpha},\varphi_{\mathcal{L}}),$ since this determines the boundary of the curve $C_{\mathcal{L}}$.  Call a representative of the isomorphism class of $(\mathfrak{s}_{\alpha},\varphi_{\mathcal{L}})$ $\mathcal{L}${\em -compatible} if the boundary of $C_{\mathcal{L}}$ is $\mathcal{L}$.  Let $N$ denote the normal bundle of $C_{\mathcal{L}}$.  Given an $\mathcal{L}$-compatible representative, projection induces a canonical isomorphism between $\xi|_{\partial C_{\mathcal{L}}}$ and $N|_{\partial C_{\mathcal{L}}}$ and the trivialization $\tau$ induces a trivialization of $N$ over $\partial C_{\mathcal{L}}$.  Remembering that $K^{-1}|_Y = \xi$, we can therefore follow \cite{hutchings_absolute} and define $c_1(N,\tau)$ (resp. $c_1(K^{-1}|_{C_{\mathcal{L}}},\tau))$ to be a signed count of the zeroes of a generic section of $N$ (resp. $K^{-1}|_{C_\mathcal{L}})$ extending a nonzero section over $\partial C_{\mathcal{L}}$ that has winding number $0$ with respect to $\tau.$     
     
We now have the following lemma:
\begin{lemma}
\label{lem:windingnumberzero}
There exists an $\mathcal{L}$-compatible representative for the isomorphism class of $(\mathfrak{s}_{\alpha},\varphi_{\mathcal{L}})$ and a choice of $\Psi$ extending $\varphi_L$ for which
\[ \tilde{e}(S^+_{\alpha},\varphi_{\mathcal{L}}) = c_1(N|_{C_{\mathcal{L}}},\tau)+c_1(K^{-1}|_{C_{\mathcal{L}}},\tau).\] 
\end{lemma}
\begin{proof}
The number $\tilde{e}(S^+,\varphi)$ is a signed count of the zeroes of $\Psi$.  A signed zero of $\Psi$ is precisely a signed zero of $\tilde{\gamma}$ over $C_{\mathcal{L}}$.  Now observe that $d\gamma$ induces an isomorphism $N\stackrel{\simeq}{\to}E|_{C_\mathcal{L}}$, and hence the trivialization of $N$ over $\partial C_{\mathcal{L}}$ induces a trivialization of $E$ over $\partial C_{\mathcal{L}}$.  We will arrange it so that
\begin{equation}
\label{eqn:goal}  
\tilde{\gamma}=e\otimes k,
\end{equation}
where $e$ is a section of $E|_{C_{\mathcal{L}}},$ $k$ is a section of $K^{-1}|_{C_{\mathcal{L}}},$ and $e|_{\partial C_{\mathcal{L}}}$ and $k|_{\partial C_{\mathcal{L}}}$ both having winding number $0$ with respect to $\tau$. The lemma will then follow after a sign check.

To arrange for \eqref{eqn:goal}, we need to analyze the boundary of $C_{\mathcal{L}}$.  Begin by letting $U_j$ be a tubular neighborhood of one of the components for one of the $\zeta_i$; assume that $U_j$ is small enough so that $U_j$ does not contain any other components of any of the $\zeta_i$.  Recall from \S\ref{sec:ECHAbsoluteGrading} that there is a trivialization of $TU_j$ extending the trivialization $\tau$ such that the Reeb vector field is always given by $\langle 1,0,0 \rangle$ and $\xi$ is given by $\lbrace 0 \rbrace \oplus \C$ according to this trivialization.  Recall from \S\ref{sec:TaubesPerturbation} the definition of the Riemannian metric determined by the contact form and the almost complex structure.  By choosing a new representative of the homotopy class of $\tau$ if necessary, we can ensure that the Riemannian metric is given by the standard dot product in this trivialization. 

We will now choose a $\mathcal{L}$-compatible representative for $P_{\tau}(\mathcal{L})$. Let $(t,r,\theta)$ denote coordinates on $U_j$, and use the above trivialization to regard a vector field over $U_j$ as a function with values in $\R \oplus \R^2$.  Define a vector field $P_j$ in $(t,r,\theta)$ coordinates by
\begin{equation}
\label{eqn:vectorfielddefinition}
P_j(t,re^{i\theta})=(-cos(\pi r),sin(\pi r)cos(\theta),-sin(\pi r)sin(\theta)),
\end{equation}           
and extend the $P_j$ by the Reeb vector field to a vector field $P$ on $Y$.  Because the $P_j$ satisfy the conditions described in \S\ref{sec:ECHAbsoluteGrading}, the $2$-plane field $\tilde{\xi}$ corresponding to $P$ represents the homotopy class of $P_{\tau}(\mathcal{L})$.  

We then have $S^+_{\alpha} = \C \oplus \tilde{\xi}$ with $\varphi = (1,0)$.  Take $\tilde{\xi}$ to be the orthogonal complement of $P$.  Remember that $E$ is by definition the $+i$ eigenspace of Clifford multiplication by the Reeb field and $EK^{-1}$ is the $-i$ eigenspace.  To prove the lemma, we therefore need to understand the Clifford multiplication $\rho$.  Recall from the proof of \cite[Lem. 28.1.1]{km} that the Clifford multiplication is determined by requiring that $\C$ is the $+i$ eigenspace of Clifford multiplication by $P$, $\tilde{\xi}$ is the $-i$ eigenspace, and, for any vector $v$ orthogonal to $P$, $\rho(v)(\varphi)=(0,v)$. 

In particular, away from the $U_j$, the $E$ component of $\varphi$ is everywhere nonzero.  The boundary of $C_{\mathcal{L}}$ is therefore contained in the union of the $U_j$.  Restrict to a single $U_j$.  To understand the components of $\varphi$ in an eigenbasis for $\rho(R)$, it is convenient to define the vector field:  
\[
\tilde{P_j}(t,r,\theta)=(sin(\pi r),cos(\pi r)cos(\theta),-cos(\pi r)sin(\theta)).
\]

Observe that $\tilde{P_j}$ and ${P_j}$ are orthogonal, and moreover
\[
\langle 1,0,0 \rangle = -cos(\pi r) P_j + sin(\pi r) \tilde{P_j}.
\]

Because $\tilde{P_j}$ is orthogonal to $P_j$, $\tilde{P_j}$ also defines a section of $\tilde{\xi}$ over $U_j$.  We can therefore view $\lbrace \varphi, (0,\tilde{P_j}) \rbrace$ as a frame for $S^+_{\alpha}$ over $U_j$, and in this frame, Clifford multiplication by the Reeb vector field is given by
\begin{equation}
\label{eqn:thematrix}
\rho(R) = \begin{pmatrix} -icos(\pi r) & -sin(\pi r) \\ sin(\pi r) & i cos(\pi r) \end{pmatrix}. 
\end{equation}

Observe first of all that $\varphi=(1,0)$ is in the $-i$ eigenspace of $\rho(R)$ precisely when $r=0$.  This implies that the boundary of $\partial C_{\mathcal{L}}$ is $\mathcal{L}$.  Since $\rho(R)$ does not depend on $t$, we can arrange for \eqref{eqn:goal} with $e|_{\partial C_{\mathcal{L}}}$ and $k|_{\partial C_{\mathcal{L}}}$ in fact constant with respect to $\tau$.  The lemma now follows.
\end{proof}   

We can now show \eqref{eqn:laststep}, completing the proof of Theorem~\ref{thm:MainTheorem}.  Hutchings' argument from \cite[Prop. 3.1]{hutchings_original} gives
\begin{equation}
\label{eqn:eqn3}
c_1(N,\tau) = -\omega_{\tau}(\mathcal{L})+Q_{\tau}(Z_{\alpha}),
\end{equation}
and we also know that 
\begin{equation}
\label{eqn:eqn4}
c_\tau(Z_\alpha) = c_1(K^{-1}|_{C_\alpha},\tau). 
\end{equation}

Equation \eqref{eqn:laststep} now follows by choosing an $\mathcal{L}$-compatible representative and then applying Lemma~\ref{lem:windingnumberzero}, \eqref{eqn:eqn3}, \eqref{eqn:eqn4}, and the definition of $I_{ECH}$.  This completes the proof of Theorem~\ref{thm:MainTheorem}.

\paragraph{Acknowledgments} The author was partially supported
by NSF grant DMS-0838703.  The author would also like to thank Cliff Taubes, Peter Kronheimer,  Michael Hutchings, and Daniel Pomerleano for many helpful conversations.  The author would also like to thank Michael Hutchings for patiently reading over earlier drafts of this work and suggesting improvements.  Finally, the author would like to thank the anonymous reviewer for his or her suggestions.

\end{document}